\documentclass[10pt, a4paper]{article}

\usepackage{bbm}
\usepackage{amssymb,amsmath,amsthm}
\usepackage{verbatim}
\usepackage{setspace}

\normalsize \evensidemargin=0pt \oddsidemargin=0pt \hoffset=6pt
\newcommand\smx[1]{\left(\begin{smallmatrix}#1\end{smallmatrix}\right)}

\newtheorem{thm}{Theorem}[section]
\newtheorem{prop}{Proposition}[section]
\newtheorem*{exm}{Example}

\usepackage{hyperref}

\begin{document}

\title{On the extendibility of partially and Markov exchangeable binary sequences}

\author{DAVIDE DI CECCO\footnote{email: {\tt davide.dicecco@gmail.com}}}
\date{}
\maketitle

\begin{abstract}
In [Fortini et al., Stoch. Proc. Appl. {\bf 100} (2002), 147--165]
it is demonstrated that a recurrent Markov exchangeable process in
the sense of Diaconis and Freedman  is essentially a partially
exchangeable process in the sense of de~Finetti. In case of finite
sequences there is not such an equivalence. We analyze both finite
partially exchangeable and finite Markov exchangeable binary
sequences and formulate necessary and sufficient conditions for
extendibility in both cases.
\end{abstract}

\section{\large Introduction}

A finite  sequence of r.v.s  $(X_1,\ldots,X_n)$  defined on a
common probability space is said exchangeable (sometimes
$n$--exchangeable) if its joint distribution is invariant under
permutations of its components. The sequence  may  or may not be
the initial segment of a longer exchangeable sequence, i.e., as is
said, it may or may not be ``extendible'', and is said
$\infty$--extendible, if it is the initial segment of an infinite
exchangeable  sequence. de~Finetti characterized all the
$n$--exchangeable sequences of r.v.s taking values in a finite
space $I$, disregarding their extendibility,  as unique mixtures
of certain $n$--exchangeable, not extendible distributions, namely
the hypergeometric processes. From this result, he has been able
to demonstrate his representation theorem for exchangeable
infinite sequences by a passage to the limit, and in
\cite{DeFinetti69} derived necessary and sufficient conditions for
extendibility of $\{0,1\}$--valued finite sequences in a geometric
approach (see also \cite{Diaconis77},
\cite{DiaconisFreedman80finite}, \cite{Crisma82} and
\cite{Wood92}).

Under partial exchangeability, introduced in \cite{DeFinetti38}
(pp. 193--205 of \cite{Jeffrey80}), $(X_1$,$\ldots$,$X_n)$ is
divided into groups or subsequences (e.g. women and men)
accordingly to a characteristic we consider relevant (e.g. each
unit's sex), and we retain exchangeability to hold just for
variables within the same subsequence. Again, we can represent
every finite partially exchangeable sequence as a mixture of not
extendible, partially exchangeable sequences, and an analogous
representation theorem holds if all the exchangeable subsequences
forming it  are $\infty$--extendible. de~Finetti in
\cite{DeFinetti59} (pp. 147--227 of \cite{DeFinetti72}), suggested
to consider in a sequence of observations  the last observation
preceding the present one as a relevant characteristic to define a
an interesting case of partial exchangeability. Consider a finite
state space $I$; call the variables immediately subsequent any
occurrence of $i \in I$ the successors of $i$. Then the
subsequences forming the partially exchangeable sequence are those
constituted of the successors of each state in $I$. He apparently
suggested the possibility to characterize, by the usual passage to
the limit, the mixtures of Markov Chains processes as partially
exchangeable processes of that kind.

Diaconis and Freedman in  \cite{DiaconisFreedman80MC} demonstrated
that the limit argument does not hold for mixtures of transient
Markov Chains. They dropped the intuitive idea  of ``relevant
characteristics'', introduce a different notion of partial
exchangeability in terms of sufficient statistics (we will call
this case Markov exchangeability) and characterized the mixtures
of Markov Chains under the additional assumption of recurrence of
the process. In \cite{Fortinietal02} it is demonstrated that the
two definitions (that in terms of subsequences and that in terms
of sufficient statistics) coincide in case of recurrent processes.
But they differ in case of finite sequences.

We will focus on partial exchangeability in the sense of
de~Finetti and Markov exchangeability for finite sequences of
$\{0,1\}$--valued variables, and on the respective notions of
extendibility. Some necessary conditions for the extendibility of
a partially exchangeable finite sequence have been studied in
\cite{VonPlato91}  and in \cite{ScarsiniVerdicchio93}. Finite
Markov exchangeable sequences have been analyzed in \cite{Zaman84,
Zaman86}, but, as far as I know, no criterion for extendibility in
the Markov exchangeable case has been given. In Section \ref{sec:
def} we define a general framework in order to analyze this topic.
In Sections \ref{sec: par exb} we analyze the partially
exchangeable case. In particular we present two bijective
transformations of the probabilities defining a binary partially
exchangeable distribution (i.e. two alternative
parameterizations). The first, introduced by de~Finetti, allows us
to establish necessary and sufficient conditions for extendibility
developing the geometric approach presented in \cite{DeFinetti69},
\cite{Diaconis77} and \cite{Crisma82} for the simply exchangeable
case. The second parameterization, in terms of generalized
covariances, allows us to derive some simpler necessary conditions
related to the central moments of the mixing distributions. In
Section \ref{sec: mark exb} we formulate analogous results for
Markov exchangeable distributions.

\section{\large A general setting} \label{sec: def}

A sequence partially exchangeable in the sense of de~Finetti is
essentially a set of distinct exchangeable subsequences. The
concept of partial exchangeability has been extended in various
way, relating to  ergodic theory and extreme points representation
of a convex set, (see \cite{Freedman62}, \cite{Dynkin78},
\cite{DiaconisFreedman84}, \cite[chap. 12]{Aldous85}). In our case
of discrete time processes taking values in a finite state space,
we will refer to a simple formalization in terms of sufficient
statistics borrowed from \cite{DiaconisFreedmanKoch}, (see also
\cite{Zaman84}). With this formalization, we can represent also
simple exchangeability and partial exchangeability in the sense of
de~Finetti.

Let $(\Omega, \mathcal{F}, P)$ be the probability space on which
all the r.v.s in the sequel will be defined. Consider a sequence
of $n$ r.v.s $(X_1,\ldots ,X_n)$  each  taking values in a finite
set $I$. Consider a statistic $T$ from $I^n$ into a finite set
$\{t_1,\ldots , t_z\}$. We call the  sequence, as well as its
joint distribution, $n$--partially exchangeable with respect to
$T$ if:
\begin{equation}\label{eq: T partexb}
T(\textbf{x}_1) = T(\textbf{x}_2) \Rightarrow P(\textbf{x}_1) =
P(\textbf{x}_2) \qquad \forall \: \textbf{x}_1,\textbf{x}_2 \in
I^n
\end{equation}
That is, $T$ induces a partition of $I^n$ into $z$ equivalence
classes and $P$ attributes the same probability to the elements
within the same class. So we can say that $T$ is a minimal
sufficient statistic for $(X_1,\ldots ,X_n)$ under $P$. Denote
with $[t_i]$ the set $\{\textbf{x} \in I^n \; : \;
T(\textbf{x})=t_i\}$; denote $P(\textbf{x} \in [t_i])$ as
$w\,_{t_i}$, and the probability of any specified sequence in
$[t_i]$ as $p\,_{t_i}$. We have $w\,_{t_i}= |\:[t_i]\:| \cdot
p\,_{t_i}$ where $|\:[t_i]\:|$ denotes the cardinality of the set
$[t_i]$, and the distribution of $(X_1,\ldots ,X_n)$ is completely
defined by the $z$ probabilities $(w\,_{t_1},\ldots ,w\,_{t_z})$
subjected to $\sum_{i=1}^z w\,_{t_i}=1$. On the converse, any set
of nonnegative values $(w\,_{t_1},\ldots ,w\,_{t_z})$ having sum
1, defines a sequence $n$--partially exchangeable w.r.t. $T$.
Consequently the space
\begin{equation}\label{eq: z-simplex}
\diamondsuit_z = \left\{(w\,_{t_1},\ldots ,w\,_{t_z}) \; : \;
w\,_{t_i}\geq 0  ,\:   i=1, \ldots , z ,\: \sum_{i=0}^z
w\,_{t_i}=1 \right\}
\end{equation}
which is the $(z-1)$--dimensional unitary simplex embedded in
$\mathbb{R}^z$, represents all the distributions $n$--partially
exchangeable w.r.t. $T$.  Let $h\,_{[t_i]}(\textbf{x}) =
P(\textbf{x} \; |\; T(\textbf{x})=t_i)$ be  the conditional
probability distribution on $I^n$ given $T$, assessing equal
masses to all the sequences in the equivalence class $[t_i]$ and
mass 0 to the other sequences:
\[
h\,_{[t_i]}(\textbf{x}) =\left\{%
\begin{array}{ll}
  1/|\:[t_i]\:| & \quad \mbox{ if } \textbf{x} \in [t_i] \\
  0 & \quad \mbox{ otherwise } \\
\end{array}%
\right.
\]
Then the following stated in \cite{DiaconisFreedmanKoch} is plain:
\begin{thm}[\cite{DiaconisFreedmanKoch}]\label{thm:finite partial exb}
The set of all the distributions over $I^n$  partially
exchangeable w.r.t. $T$ is a simplex whose vertices are the
extremal distributions $h\,_{[t_i]}$, $i=1,\ldots ,z$, and each
partially exchangeable distribution is a unique mixture of those
extremal distributions with mixing weights
$w\,_{t_1},\ldots,w\,_{t_z}$.
\end{thm}

The extremal distributions can be conceived as urn processes
without replacement, and, depending on the properties of $T$,
de~Finetti's style theorems may be deduced by the convergence of
the hypergeometric processes to the i.i.d. processes.

\section{\large Partially exchangeable binary sequences in the sense of de~Finetti} \label{sec: par
exb}

We say that $(X_1,\ldots ,X_n)$ is partially exchangeable in the
sense of de~Finetti of order $(n_1, \ldots ,n_g)$, and we will
denote it $(n_1, \ldots ,n_g)$--DFPE, if it can be divided into
$g$ exchangeable subsequences $(X_{i,1}, \ldots ,X_{i,n_i})$,
$i=1, \ldots ,g$, $\sum_i n_i=n$. Denote $\sum_{j=1}^{n_i}X_{i,j}$
as $S_{i}$. If the variables are $\{0,1\}$--valued,  $(S_1, \ldots
,S_g)$ is a sufficient statistic in the sense of \eqref{eq: T
partexb}. Denote $P\left(\textbf{x} \in I^n \: : \: S_1=k_1,
\ldots ,S_g=k_g\right)$ as $w\,_{k_1, \ldots ,k_g}^{(n_1, \ldots
,n_g)}$ and the probability of any sequence consistent with
$\left(S_1=k_1, \ldots ,S_g=k_g\right)$ as $p\,_{k_1, \ldots
,k_g}^{(n_1, \ldots ,n_g)}$. Then we have:
\begin{equation}\label{eq: partialp partialq}
w\,_{k_1, \ldots ,k_g}^{(n_1, \ldots ,n_g)} =
\binom{n_1}{k_1}\cdots \binom{n_g}{k_g} \; p\,_{k_1, \ldots
,k_g}^{(n_1, \ldots ,n_g)}
\end{equation}
An $(n_1, \ldots ,n_g)$--DFPE distribution is defined by the
$(n_1+1) \cdots (n_g+1)$ probabilities $w\,_{k_1, \ldots
,k_g}^{(n_1, \ldots ,n_g)}$ defined for every $g$--tuple of
nonnegative integers $(k_1, \ldots ,k_g)$ such that $0\leq k_i
\leq n_i$ for $i=1, \ldots ,g$, subjected to
\[
\sum_{k_1=0}^{n_1} \cdots \sum_{k_g=0}^{n_g} w\,_{k_1, \ldots
,k_g}^{(n_1, \ldots  ,n_g)} =  1
\]
For what we have said  in \eqref{eq: z-simplex}, the
$\left\{w\,_{k_1, \ldots ,k_g}^{(n_1, \ldots
,n_g)}\right\}_{\substack{k_i\leq n_i\\ i=1,\ldots,g}}$ range in
the unitary simplex $\diamondsuit_{(n_1+1) \cdots (n_g+1)}$.

The exchangeability of each subsequence
$(X_{i,1},\ldots,X_{i,n_i})$ implies the exchangeability  of all
its subsets, and we can obtain  all the probabilities of the kind
$\left\{w\,_{l_1, \ldots ,l_g}^{(m_1, \ldots
,m_g)}\right\}_{\substack{l_i\leq m_i\\ i=1,\ldots,g}}$,
$m_i<n_i$, from the $\left\{w\,_{k_1, \ldots ,k_g}^{(n_1, \ldots
,n_g)}\right\}_{\substack{k_i\leq n_i\\ i=1,\ldots,g}}$ through
the following easily proved formula:
\begin{equation}\label{eq: partialq partialq}
w\,_{l_1,...,l_g}^{(m_1,...,m_g)} = \sum_{k_1=l_1}^{n_1-m_1+l_1}
\cdots \sum_{k_g=l_g}^{n_g-m_g+l_g} \frac{\binom{k_1}{l_1}
\binom{n_1-k_1}{m_1-l_1}}{\binom{n_1}{m_1}}\cdots
\frac{\binom{k_g}{l_g}\binom{n_g-k_g}{m_g-l_g}}{\binom{n_g}{m_g}}\:\:
w\,_{k_1,...,k_g}^{(n_1,...,n_g)}
\end{equation}
Denote in particular  the probabilities $w\,_{k_1, \ldots
,k_g}^{(k_1, \ldots ,k_g)}$ as $w\,_{k_1, \ldots ,k_g}$.  We have
\begin{equation}\label{def partiallambda}
P\left\{\bigcap_{i=1}^g (X_{i,s_1}=1, \ldots
,X_{i,s_{k_i}}=1)\right\} = E\left[\prod_{i=1}^g X_{i,s_{1}}
\cdots X_{i,s_{k_i}}\right] = w\,_{k_1,\ldots,k_g}
\end{equation}
for every subset $(s_{1}, \ldots  ,s_{k_i})$ of $k_i$ labels in
$\{1,\ldots,n_i\}$, $i=1,\ldots, g$.  By \eqref{eq: partialq
partialq} we have
\begin{equation}\label{eq: partialq partiall}
w\,_{k_1, \ldots  ,k_g} = \sum_{i_1=k_1}^{n_1} \cdots
\sum_{i_g=k_g}^{n_g} \frac{(i_1)_{k_1}}{(n_1)_{k_1}} \cdots
\frac{(i_g)_{k_g}}{(n_g)_{k_g}} \;\; w\,_{i_1, \ldots ,i_g}^{(n_1,
\ldots  ,n_g)}
\end{equation}
where, from now on, $(i)_k =i(i-1)\cdots (i-k+ 1)$ for $k\leq i$
and $(i)_0=1$.

To define the inverse map of \eqref{eq: partialq partiall}
introduce the difference operator $\Delta_i$ w.r.t. the $i$-th
group: $\Delta_i \left(w\,_{k_1, \ldots ,k_g}\right) = w\,_{k_1,
\ldots ,k_i+1, \ldots  ,k_g} - w\,_{k_1, \ldots ,k_i, \ldots
,k_g}$. Then  we have (see \cite{DeFinetti38})
\begin{equation}\label{eq: partiall partialq}
w\,_{k_1, \ldots  ,k_g}^{(n_1, \ldots  ,n_g)} =
\binom{n_1}{k_1}\cdots \binom{n_g}{k_g} (-1)^{n_1-k_1 + \ldots   +
n_g -k_g} \;\; \Delta_1^{n_1-k_1} \cdots \Delta_g^{n_g-k_g}
\big(w\,_{k_1, \ldots  ,k_g}\big)
\end{equation}
Where $w\,_{0, \ldots ,0}=1$. So the $\{w\,_{k_1, \ldots ,k_g}\}_{\substack{k_i\leq n_i\\
i=1,\ldots,g}}$  suffice to completely define any $(n_1, \ldots
,n_g)$--DFPE binary sequence, i.e. they constitute a
parameterization of an $(n_1, \ldots ,n_g)$--DFPE binary
distribution.

By \eqref{eq: partiall partialq}, in an $(n_1,\ldots ,n_g)$--DFPE
sequence each probability $w\,_{k_1, \ldots , k_g}$ should satisfy
\begin{equation}\label{eq: conditions partiall}
(-1)^{n_1-k_1 + \ldots   + n_g -k_g} \;\; \Delta_1^{n_1-k_1}
\cdots \Delta_g^{n_g-k_g} \big(w\,_{k_1, \ldots  ,k_g}\big) \geq 0
\end{equation}
Moreover, since by \eqref{eq: partialq partiall} it is
$\sum_{k_1=0}^{n_1} \cdots \sum_{k_g=0}^{n_g} w\,_{k_1, \ldots
,k_g}^{(n_1, \ldots ,n_g)} = w\,_{0,\ldots,0} = 1$, the \eqref{eq:
conditions partiall} constitute necessary and sufficient
conditions for a set $\{w\,_{k_1, \ldots , k_g}\}_{\substack{k_i\leq n_i\\
i=1,\ldots,g}}$ with $w\,_{0,\ldots,0} = 1$ to define an
$(n_1,\ldots ,n_g)$--DFPE sequence. Then the $\{w\,_{k_1, \ldots , k_g}\}_{\substack{k_i\leq n_i\\
i=1,\ldots,g}}$ range in the space
\[
\Lambda_{n_1, \ldots  ,n_g}=\Big\{(w\,_{k_1, \ldots , k_g})_{\substack{k_i\leq n_i\\
i=1,\ldots,g}} \:,\: \sum_i k_i >0 \: : \;\mbox{satisfy \eqref{eq:
conditions partiall}} \Big\}
\]

\subsection{\textmd{Generalized covariances}}

We introduce a generalization of the usual concept of covariance
defined as follows: the covariance of order $k$ among the
variables $X_1, \ldots  ,X_k$ is
\begin{equation}\label{eq: generalized covariance}
Cov[X_1, \ldots  ,X_k] = E[(X_1-E[X_1]) \cdots (X_k-E[X_k])]
\end{equation}
Under DFPE, these covariances depends only on the number of
variables involved for each exchangeable subsequence. To simplify
the notation,  denote the value $w\,_{k_1,\ldots ,k_g}$ when
$k_i=1$ and all other subscripts are zero as $w(i)$, i.e.
$E[X_{i,1}] = w(i)$. Then under DFPE any generalized covariance
involving $k_i$ r.v.s of the $i$--th subsequence, $i=1,\ldots ,g$
is equal to
\[
Cov\,_{k_1, \ldots  ,k_g} = E \Big[\big(X_{1,1} - w(1)\big) \cdots
\big(X_{1,k_1} - w(1)\big) \cdots \big(X_{g,1} - w(g)\big) \cdots
\big(X_{g,k_g} - w(g)\big) \Big]
\]
and the relation with the previous parameterization is
\begin{equation}\label{eq: partiall partialcov}
Cov\,_{k_1, \ldots  ,k_g} = \sum_{i_1=0}^{k_1} \cdots
\sum_{i_g=0}^{k_g} \binom{k_1}{i_1} \cdots \binom{k_g}{i_g}
(-1)^{i_1+ \ldots  +i_g} \big(w(1)\big)^{i_1} \cdots
\big(w(g)\big)^{i_g} \;\:w\,_{k_1-i_1, \ldots  ,k_g - i_g}
\end{equation}
\textit{Proof of \eqref{eq: partiall partialcov}:} For the sake of
simplicity, but without loss of generality, set g=2. By expanding
the product,
\[
\big(X_{1,1} - w(1)\big) \cdots \big(X_{1,k_1} - w(1)\big)
\big(X_{2,1} - w(2)\big) \cdots \big(X_{2,k_2} - w(2)\big)
\]
results as the sum of $(k_1+1)(k_2+1)$ terms of the kind
\begin{equation}\label{eq: partialcov passage}
\sum_{h_1<\ldots <h_i} \sum_{s_1<\ldots <s_j}
\big(-w(1)\big)^{k_1-i}  \big(-w(2)\big)^{k_2-j} X_{1,h_1}\cdots
X_{1,h_i}\cdot X_{1,s_1}\cdots X_{1,s_j}
\end{equation}
where the first sum ranges over all the possible $i$--tuples
$(h_1,\ldots ,h_i)$ of distinct labels in $\{1,\ldots,n_1\}$ and
consists of $\binom{k_1}{i}$ terms, the second of $\binom{k_2}{j}$
terms. Passing to the expectation, by \eqref{def partiallambda},
the term \eqref{eq: partialcov passage} results as $\binom{k_1}{i}
\binom{k_2}{j}\big(-w(1)\big)^{k_1-i} \big(-w(2)\big)^{k_2-j}
w\,_{i,j}$, so that
\[
Cov\,_{k_1,k_2} = \sum_{i=0}^{k_1} \sum_{j=0}^{k_2} \binom{k_1}{i}
\binom{k_2}{j}(-1)^{k_1+k_2-i-j} w(1)^{k_1-i} w(2)^{k_2-j}
w\,_{i,j}
\]
\hfill $\Box$ \linebreak \\

One can prove that the inverse map, which is somewhat similar to
the inverse of a binomial transform, is
\begin{equation}\label{eq: partialcov partiall}
w\,_{k_1, \ldots  ,k_g} = \sum_{i_1=0}^{k_1} \cdots
\sum_{i_g=0}^{k_g} \binom{k_1}{i_1} \cdots \binom{k_g}{i_g}
w(1)^{i_1} \cdots w(g)^{i_g} \;\: Cov\,_{k_1-i_1, \ldots ,k_g -
i_g}
\end{equation}
where $Cov\,_{0, \ldots  ,0}=1$ and all the covariances having a
single 1 and all zeros in the subscript are zero. So, a
$(n_1,\ldots ,n_g)$--DFPE binary sequence is completely defined by
the  $g$ probabilities $w(1), \ldots ,w(g)$  together with the
generalized covariances $\{Cov\,_{k_1,\ldots,k_g}\}$ defined for
every $g$--tuple $(k_1,\ldots,k_g)$ with $k_i\leq n_i$ and such
that $\sum_{i=1}^g k_i \geq 2$. The space of the
$Cov\,_{k_1,\ldots,k_g}$ is implicitly defined by
$\Lambda\,_{n_1,\ldots,n_g}$ and \eqref{eq: partiall partialcov}
and is not easily described. We can say that all the
$Cov\,_{k_1,\ldots,k_g}$ can be both positive or negative, and by
\eqref{eq: partialcov partiall} are all null if, and only if,
$X_1, \ldots , X_n$ are i.i.d.

\subsection{\textmd{Extendibility}}\label{sec: ext part}

For the sake of simplicity in this section we set $g=2$, but all
the results hold for a general $g$.

For what we have said, we can represent any $(n_1,n_2)$--DFPE
distribution as a point in the linear spaces
$\diamondsuit_{(n_1+1)(n_2+1)}$ and $\Lambda_{n_1, n_2}$. Formulas
\eqref{eq: partiall partialq} and \eqref{eq: partialq partiall}
define the linear maps between the two  spaces. Clearly these maps
are one--one and onto and establish affine congruence of the two
sets. The $(n_1+1)(n_2+1)$ vertices of
$\diamondsuit_{(n_1+1)(n_2+1)}$ are the points having one
coordinate equal to one and the others equal to zero and represent
the extremal distributions of Theorem \ref{thm:finite partial
exb}. \eqref{eq: partialq partiall} maps this vertices  onto the
vertices of $\Lambda_{n_1, n_2}$. In particular, the extremal
distribution having $w\,_{k_1, k_2}^{(n_1, n_2)}=1$ is represented
in $\Lambda_{n_1, n_2}$ by the point
$\lambda\,_{k_1,k_2  ; \; n_1,n_2}$ $\equiv$ $(w\,_{l_1, l_2})\,_{\substack{l_1\leq n_1\\
l_2\leq n_2}}$,  having coordinates
\[
w\,_{l_1, l_2} = \left\{
\begin{array}{ll}
0 & \mbox{whenever } l_i > {k}_i \mbox{ for any } i=1, 2\\
\displaystyle\frac{({k}_1)_{l_1} ({k}_{2})_{l_{2}} }{({n}_1)_{l_1}
({n}_{2})_{l_{2}}} & \mbox{elsewhere }
\end{array}
\right.
\]
The points $\left\{ \lambda\,_{k_1,k_2  ; \; n_1,n_2}
\right\}_{\substack{k_1\leq n_1\\ k_2\leq n_2}}$, are affinely
independent, then  $\Lambda_{n_1,n_2}$, which is their convex
hull, is a $(n_1+1)(n_2+1)-1$ dimensional convex polytope with
$(n_1+1)(n_2+1)$ vertices, i.e. a non--standard simplex.

We say that a $(n_1, n_2)$--DFPE sequence is (at least)
$(r_1,r_2)$--extendible, $r_i\geq n_i$, if it is the initial
segment of a $(r_1, r_2)$--DFPE sequence. So the sequence,
represented by the point $w \equiv (w\,_{l_1,
l_2})\,_{\substack{l_1\leq n_1\\ l_2\leq n_2}}$ in $\Lambda_{n_1,
n_2}$, is $(r_1,r_2)$--extendible if, and only if, there exist a
point $w^* \equiv (w\,_{k_1, k_2})\,_{\substack{k_1\leq r_1\\
k_2\leq r_2}}$ in $\Lambda_{r_1, r_2}$ such that its orthogonal
projection over the coordinates of $\Lambda_{n_1, n_2}$ coincide
with $w$. That is, denote as $\Lambda_{r_1, r_2}^{(n_1, n_2)}$ the
projection of $\Lambda_{r_1, r_2}$ over the coordinates of
$\Lambda_{n_1, n_2}$, and as $\lambda\,_{k_1,k_2  ; \;
r_1,r_2}^{(n_1,n_2)}$ the analogous projection of
$\lambda\,_{k_1,k_2  ; \; r_1,r_2}$. Then $\Lambda_{r_1,
r_2}^{(n_1, n_2)}$ is exactly the subspace of $\Lambda_{n_1, n_2}$
representing the $(n_1,n_2)$--DFPE distributions which are at
least $(r_1, r_2)$--extendible and it results as the convex hull
of the $\left\{ \lambda\,_{k_1,k_2  ; \; r_1,r_2}^{(n_1,n_2)}
\right\}_{\substack{k_1\leq r_1\\ k_2\leq r_2}}$. Moreover, we are
going to see that none of this point is redundant with respect to
the convex hull problem, that is they are exactly the vertices of
$\Lambda_{r_1,  r_2}^{(n_1, n_2)}$.
\begin{thm}\label{thm: representation partiall}
\begin{equation}\label{eq: representation partiall}
\begin{split}
\lambda\,_{k_1,k_2 \,;\, r_1,r_2}^{(n_1,n_2)} & =  \frac{r_1-k_1}{r_1} \; \lambda\,_{k_1,k_2 \,;\, r_1-1,r_2}^{(n_1,n_2)} + \frac{k_1}{r_1} \; \lambda\,_{k_1-1,k_2 \,;\, r_1-1,r_2}^{(n_1,n_2)}\\
                                              & =  \frac{r_2-k_2}{r_2} \; \lambda\,_{k_1,k_2 \,;\, r_1,r_2-1}^{(n_1,n_2)} + \frac{k_2}{r_1} \; \lambda\,_{k_1,k_2-1 \,;\, r_1,r_2-1}^{(n_1,n_2)}
\end{split}
\end{equation}
\end{thm}
\begin{proof}
We have
\begin{align}
p\,_{k_1,k_2}^{(n_1,n_2)} & = p\,_{k_1,k_2}^{(n_1+1,n_2)} + p\,_{k_1+1,k_2}^{(n_1+1,n_2)}\label{eq: partialp partialp 1}\\
                          & = p\,_{k_1,k_2}^{(n_1,n_2+1)} + p\,_{k_1,k_2+1}^{(n_1,n_2+1)}\label{eq: partialp partialp 2}
\end{align}
The point $\lambda_{k_1,k_2 \,;\, r_1,r_2}$ represents the
distribution having $w\,_{k_1,k_2}^{(r_1,r_2)}=1$. Any term
$p\,_{k_1,k_2}^{(r_1,r_2)}$ appears  in the right hand side of
exactly one equation of the kind \eqref{eq: partialp partialp 1}
and one of the kind  \eqref{eq: partialp partialp 2}. Then, by
\eqref{eq: partialp partialq} it is easily seen that if
$w\,_{k_1,k_2}^{(r_1,r_2)}=1$, it is
\[
w\,_{k_1,k_2}^{(r_1-1,r_2)} = \frac{r_1-k_1}{r_1},\:
w\,_{k_1-1,k_2}^{(r_1-1,r_2)} = \frac{k_1}{r_1},\:
w\,_{k_1,k_2}^{(r_1,r_2-1)} = \frac{r_2-k_2}{r_2},\:
w\,_{k_1,k_2-1}^{(r_1,r_2-1)} = \frac{k_2}{r_2}
\]
then the statement follows by \eqref{eq: partialq partiall}.
\end{proof}

\begin{prop}\label{prop: edges}
Consider a polytope $A$ and a set of points  lying on distinct
edges of $A$. Call $A'$ their convex hull. Then those points  are
the  vertices of $A'$.
\end{prop}
\begin{proof}
Say one of those point $v$ lies on the edge $e$ of $A$. Then it
can only be represented as  convex combinations of points in $e$,
and no other points in $A$. But $v$ is the only point of $A'$
lying on $e$ and obviously $A' \subset A$, then $v$ cannot be
represented as convex combinations of any other points in $A'$,
and hence is a vertex.
\end{proof}
\begin{thm}
\textbf{a$)$} The  $\left\{ \lambda\,_{k_1,k_2  ; \;
r_1,r_2}^{(n_1,n_2)} \right\}_{\substack{k_1\leq r_1\\ k_2\leq
r_2}}$  are the vertices of $\Lambda\,_{r_1,r_2}^{(n_1,n_2)}$.\\
\textbf{b$)$} Each pair of points in the right hand side of
\eqref{eq: representation partiall} constitute the vertices of an
edge of their own space.
\end{thm}
\begin{proof}
$\Lambda\,_{n_1,n_2}$ is a simplex, so each couple of its vertices
identifies an edge. By \eqref{eq: representation partiall}, the
points $\lambda\,_{k_1,k_2 \, ; \, n_1+1,n_2}^{(n_1,n_2)}$ of
$\Lambda\,_{n_1+1,n_2}^{(n_1,n_2)}$ lie on distinct edges of
$\Lambda\,_{n_1,n_2}$ and by Proposition \ref{prop: edges} are all
vertices of $\Lambda\,_{n_1+1,n_2}^{(n_1,n_2)}$. Moreover, each
couple of vertices of $\Lambda\,_{n_1+1,n_2}^{(n_1,n_2)}$ of the
kind $\lambda\,_{k_1,k_2\,;\, n_1+1,n_2}^{(n_1,n_2)}$,
$\lambda\,_{k_1+1,k_2\,;\, n_1+1,n_2}^{(n_1,n_2)}$ lie on two
adjacent  edges of $\Lambda\,_{n_1,n_2}$ having the vertex
$\lambda\,_{k_1,k_2\,;\, n_1,n_2}$ in common, and no other vertex
of $\Lambda\,_{n_1+1,n_2}^{(n_1,n_2)}$ has
$\lambda\,_{k_1,k_2\,;\, n_1,n_2}$ in its representation
\eqref{eq: representation partiall}. So they identify an edge of
$\Lambda\,_{n_1+1,n_2}^{(n_1,n_2)}$. To be precise, all the points
$\lambda\,_{k_1,k_2 \, ; \, n_1+1,n_2}^{(n_1,n_2)}$ having $k_1=0$
or $k_1= n_1+1$ coincide with vertices of $\Lambda\,_{n_1,n_2}$.
However, as we have said, there are not three points having a
common vertex of $\Lambda\,_{n_1,n_2}$ in their representation
\eqref{eq: representation partiall}, so they are vertices of
$\Lambda\,_{n_1+1,n_2}^{(n_1,n_2)}$ as well. In conclusion, a) and
b) are valid for $r_1=n_1+1$, and obviously also for $r_2=n_2+1$.
It is easily seen that, if we suppose a) and b) hold for
$\Lambda\,_{r_1,r_2}^{(n_1,n_2)}$, then they also hold for
$\Lambda\,_{r_1+1,r_2}^{(n_1,n_2)}$ and
$\Lambda\,_{r_1,r_2+1}^{(n_1,n_2)}$, so the theorem is proved by
induction.
\end{proof}
In conclusion, an $(n_1,n_2)$--DFPE distribution, represented by a
point $w$ in $\Lambda_{n_1,n_2}$, is at least
$(r_1,r_2)$--extendible if, and only if, $w$ is contained in
$\Lambda\,_{r_1,r_2}^{(n_1,n_2)}$, and is exactly
$(r_1,r_2)$--extendible if $\Lambda\,_{r_1+1,r_2}^{(n_1,n_2)}$ and
$\Lambda\,_{r_1,r_2+1}^{(n_1,n_2)}$ do not contain $w$.



Note that, by virtue of \eqref{eq: partialq partialq} we can  map
the extremal points of $\diamondsuit_{(r_1+1)(r_2+1)}$ and find
the subspace of $\diamondsuit_{(n_1+1)(n_2+1)}$ representing the
$(n_1,n_2)$--DFPE distribution that are at least
$(r_1,r_2)$--extendible. But the probabilities
$w\,_{k_1,k_2}^{(n_1,n_2)}$ depend on $n_1$ and $n_2$, so we
should obtain the vertices of the subspaces for each couple
$(n_1,n_2)$. On the converse, the probabilities $w\,_{k_1,k_2}$ do
not depend on the sequence size, and once we know the vertices of
$\Lambda\,_{r_1,r_2}$ we can obtain the vertices of
$\Lambda\,_{r_1,r_2}^{(n_1,n_2)}$ for every $n_1 < r_1$, $n_2 <
r_2$ simply excluding certain coordinates.

The points in $\Lambda\,_{r,0}^{(n,0)}$ represents the
$n$--exchangeable distributions that are at least $r$--extendible.
The $(n-1)$--dimensional faces of an $n$--dimensional polytope are
said facets. A polytope is said simplicial if all its facets are
simplexes. Crisma in \cite{Crisma82} demonstrated that the
$\Lambda\,_{r,0}^{(n,0)}$ are simplicial and their vertices
satisfy  Gale Evenness Condition (a combinatorial property
characterizing the facets). As a consequence, we can easily
determine if a point lies inside any $\Lambda\,_{r,0}^{(n,0)}$.
Moreover, Crisma has been able to compute their volumes,
determining in some sense the proportion of $n$--exchangeable
sequences that are $r$--extendible.

Unfortunately, the  $\Lambda\,_{r_1,r_2}^{(n_1,n_2)}$ are not
simplicial polytopes, and we have not found an analytical way to
determine their facets. Then, to determine if a point $w$ lies
inside a certain polytope $\Lambda\,_{r_1,r_2}^{(n_1,n_2)}$ we can
use the following linear program:
\begin{equation}\label{eq: LP}
\begin{split}
  \text{maximize }  & \quad z^T w -z_0  = f                               \\
  \text{subject to }& \quad z^T \lambda - z_0 \leq 0  \qquad \forall \; \lambda \in \left\{ \lambda\,_{k_1,k_2\,;\, r_1,r_2}^{(n_1,n_2)} \right\}_{\substack{k_1\leq r_1\\ k_2\leq r_2}}\\
                    & \quad z^T w -z_0 \leq 1
\end{split}
\end{equation}
where $z \in \mathbb{R}^{(n_1+1)(n_2+1)}$ and $z_0 \in
\mathbb{R}$. The last inequality is artificially added so that the
linear program has a bounded solution. The optimal value $f$ is
positive if and only if there exists an hyperplane $\{x \in
\mathbb{R}^{(n_1+1)(n_2+1)} \: : \:  z^T x = z_0\}$ separating the
polytope $\Lambda\,_{r_1,r_2}^{(n_1,n_2)}$ and $w$, i.e. if and
only if $w$ lies outside of $\Lambda\,_{r_1,r_2}^{(n_1,n_2)}$.

\subsubsection{$\infty$--extendible case}

If all the $g$  subsequences of a DFPE sequence are
$\infty$--extendible, there exists a probability measure $\nu$
over the $g$--dimensional hypercube $[0,1]^g$ and a r.v.
$\Theta=\big(\theta(1), \ldots ,\theta(g)\big)$ distributed
accordingly such that
\begin{equation}\label{eq: mixture partial}
w\,_{k_1, \ldots  ,k_g}^{(n_1, \ldots  ,n_g)} =
\binom{n_1}{k_1}\cdots \binom{n_g}{k_g} \int_0^1 \!\!\cdots
\!\int_0^1 \prod_{i=1}^g \theta(i)^{k_i}
\big(1-\theta(i)\big)^{n_i-k_i} \; d \nu (\Theta)
\end{equation}
So, the probabilities $w\,_{k_1, \ldots ,k_g}$ are the ordinary
mixed moments of the mixing measure $\nu$:
$E_{\nu}\Big[\theta(1)^{k_1} \cdots \theta(g)^{k_g}\Big]$. Let
$\mathcal{M}^{(n_1, \ldots  ,n_g)}$ be the space of the mixed
moments up to order $(n_1, \ldots ,n_g)$ of all the probability
measures over $[0,1]^g$. For what we have said,
$\{\Lambda\,_{r_1,\ldots,r_g}^{(n_1, \ldots ,n_g)}\}_{r_1,\ldots
,r_g}$ is a decreasing multisequence  of polytopes and we have
\[
\bigcap_{r_1=n_1}^{\infty}\cdots \bigcap_{r_g=n_g}^{\infty}
\Lambda\,_{r_1,\ldots,r_g}^{(n_1, \ldots ,n_g)} =
\mathcal{M}^{(n_1, \ldots  ,n_g)}
\]
As far as I know there is no practical criterion to establish if a
point of $\mathbb{R}^{(n_1+1)\cdots (n_g+1)}$ lies inside
$\mathcal{M}^{(n_1, \ldots ,n_g)}$. Then we can check some simple
necessary conditions for $\infty$--extendibility using moments'
inequalities.

Formulas \eqref{eq: partiall partialcov} and \eqref{eq: partialcov
partiall}  link the ordinary mixed moments and the central mixed
moments of a multivariate distribution (see e.g.
\cite{Johnsonetal97}, equations (34.28) (34.29)), consequently we
have:
\[
Cov\,_{k_1, \ldots  ,k_g}=
E_{\nu}\left[\big(\theta(1)-E_{\nu}[\theta(1)]\big)^{k_1} \cdots
\big(\theta(g)-E_{\nu}[\theta(g)]\big)^{k_g}\right]
\]
So, a simple necessary condition for a representation of the kind
\eqref{eq: mixture partial} to hold is
\begin{equation}\label{eq: necessary for Cov}
Cov\,_{2 k_1, \ldots  ,2 k_g}\geq 0 \qquad \forall \;
k_i=1,\ldots, \lfloor n_i/2 \rfloor, \quad i=1,\ldots,g
\end{equation}

To simplify the notation, denote as $Cov(i,j)$ the covariance
between a r.v. of the $i$--th group and one of the $j$--th group,
i.e. the value $Cov\,_{k_1, \ldots ,k_g}$ when $k_i=k_j=1$ and all
other subscripts are 0. Another simple necessary condition for
\eqref{eq: mixture partial} to hold is that, in that case,
for what we have said, $\big\{Cov(i,j)\big\}_{\substack{1\leq i \leq g\\
1\leq j \leq g}}$ is  the Variance--Covariance matrix of $\Theta$
and hence must be nonnegative definite.
\begin{exm}
The $(2,2)$--DFPE distribution defined by the following values of
$w\,_{k_1,k_2}^{(2,2)}$:
\begin{center}
\begin{tabular}{ c|c c c}
  $k_1 \backslash k_2$ & 0 & 1 & 2 \\
  \hline
  0 & $\frac{3}{16}$ & $\frac{3}{16}$ &  $0$   \\
  1 & $\frac{1}{16}$ & $\frac{3}{16}$ &  $0$   \\
  2 & $\frac{1}{16}$ &       $0$        & $\frac{5}{16}$ \\
\end{tabular}
\end{center}
by \eqref{eq: partialq partiall} leads to the following values of
$w\,_{k_1,k_2}$:
\begin{center}
\begin{tabular}{ c|c c c}
  $k_1 \backslash k_2$ & 0 & 1 & 2 \\
  \hline
  0 &         $1$        &  $\frac{1}{2}$   &  $\frac{5}{16}$ \\
  1 &    $\frac{1}{2}$ &  $\frac{23}{64}$ &  $\frac{5}{16}$ \\
  2 &    $\frac{3}{8}$ &  $\frac{5}{16}$  &  $\frac{5}{16}$ \\
\end{tabular}
\end{center}
and by \eqref{eq: partiall partialcov} we have
\[
Cov\,_{2,0}=1/8, \quad Cov\,_{0,2}=1/16, \quad Cov\,_{2,2}=1/32
\]
so \eqref{eq: necessary for Cov} is satisfied. But $Cov\,_{2,0}\:
Cov\,_{0,2} - Cov\,_{1,1}^2 = -\frac{17}{4096} < 0$, so
$\smx{Cov\,_{2,0} & Cov\,_{1,1}\\ Cov\,_{1,1} & Cov\,_{0,2} }$ is
not nonnegative definite and the distribution is not
$(\infty,\infty)$--extendible. The linear program \eqref{eq: LP}
reveals that the point of $\Lambda_{2,2}$ representing the
distribution lies in $\Lambda_{4,2}^{(2,2)}$, but not in
$\Lambda_{5,2}^{(2,2)}$ nor in $\Lambda_{2,3}^{(2,2)}$. Hence the
distribution is exactly $(4,2)$--extendible. Note that both the
2--exchangeable subsequences identified respectively by
$(w\,_{1,0}, w\,_{2,0})$ and by $(w\,_{0,1}, w\,_{0,2})$, are
$\infty$--extendible.
\end{exm}

\section{{\large Markov exchangeability}\label{sec: mark exb}}

Consider an $I$--valued  sequence $(x_1,\ldots ,x_n)$. Define its
transition counts $n\,_{i,j}$ for all $i$, $j$ in $I$ as
\[
n\,_{i,j} = \sum_{k=1}^{n-1} \mathbbm{1}_{(i,j)}(x_k, x_{k+1})
\]
and arrange them in a matrix $N=\{n\,_{i,j}\}_{i,j}$. Then, the
distribution of $(X_1,\ldots ,X_n)$ is Markov exchangeable
(hereafter ME or $n$--ME if we need to highlight the number of
variables) when the sufficient statistic $T$ in \eqref{eq: T
partexb} is the value of the first step $x_1$, together with the
transition count matrix $N$. Introduce the number of transitions
exiting from $i$: $n_{i}^+ = \sum_{j \in I} n\,_{i,j}$ and  the
number of transitions entering in $i$: $n_{i}^- = \sum_{j \in I}
n\,_{j,i}$.
\begin{prop}
Consider  an $I$--valued  sequence $(x_1,\ldots ,x_n)$. Then, it
is $x_1=x_n$ if, and only if
\begin{equation}\label{eq: 1 kind}
n^+_i = n^-_i \qquad \forall i \in I
\end{equation}
while it is $x_1\neq x_n$ if, and only if
\begin{equation}\label{eq: 2 kind}
\left\{
\begin{array}{l}
n^+_{x_1}  =  n^-_{x_1} + 1\\
n^-_{x_n} = n^+_{x_n} + 1\\
n^+_{i} = n^-_{i} \quad \mbox{for}\: i \neq x_1 \neq x_n\\
\end{array}
\right.
\end{equation}
Moreover, an integer valued matrix $N=\{n\,_{i,j}\}_{i,j}$ is a
consistent transition count matrix if, and only if, it is
irreducible and one between \eqref{eq: 1 kind} and \eqref{eq: 2
kind} is valid.
\end{prop}
\begin{proof}
Consider $H=\{(x_1,x_2), \ldots ,(x_{n-1},x_n)\}$ and let $J$  be
the set of the distinct states ($J \subseteq I$) visited by
$(x_1,\ldots ,x_n)$. We can think to $N$ as the adjacency matrix
of the directed graph $G=(J,H)$. But $G$ is Eulerian by
construction, and the result follows immediately.
\end{proof}

Denote as $[x_1,N]$  the set of all the  $I$--valued $n$--tuples
starting in $x_1$ and having a transition count $N$. Denote
$P(\textbf{x} \in [x_1,N])$ as $w\,_{x_1,N}$, and the probability
of having any specified  element of $[x_1,N]$ as $p\,_{x_1,N}$.
Denote the set of all the distinct transition count matrices of
all the $I$--valued $n$--tuples starting in $x_1$ as $\Phi(x_1,
n)$. For what we have said,  an $I$--valued $n$--ME distribution
is completely defined by the probabilities $w\,_{x_1,N}$ for $N$
ranging in $\Phi(x_1,n)$ and $x_1$ ranging in $I$ subjected to
\[
\sum_{x_1 \in I} \sum_{N \in \Phi(x_1,n)} w\,_{x_1,N} = 1
\qquad\mbox{and}\qquad \sum_{N \in \Phi(x_1,n)} w\,_{x_1,N} =
P(X_1=x_1)  \qquad \forall \:x_1 \in I
\]
The cardinality of $[x_1,N]$ was first found by Whittle in
\cite{Whittle55}. Define the matrix $B=\{b_{ij}\}_{i,j \in I}$ as
\[
b_{ij}=\left\{
\begin{array}{lc}
-n\,_{i,j}/n^+_i & \mbox{for } i\neq j \\
1 - n\,_{i,i}/n^+_i & \mbox{for } i = j \\
\end{array}
\right.
\]
By \eqref{eq: 1 kind} and \eqref{eq: 2 kind}, if we know the
starting state and the transition counts of a sequence, we also
know its ending state.
\begin{thm}[\cite{Whittle55}]\label{thm: Whittle}
The number of sequences in $[x_1,N]$ is
\[
\det(B_{x_n, x_n}) \; \frac{\prod_i n^+_i!}{\prod_{i,j}
n\,_{i,j}!}
\]
where $x_n$ is uniquely determined by $x_1$ and $N$, and where
$B_{x_n, x_n}$ is the matrix obtained by $B$ removing the
$x_n$--th row and the $x_n$--th column.
\end{thm}
Then it is $w\,_{x_1,N}= \det(B_{x_n, x_n}) \;
\displaystyle\frac{\prod_i n^+_i!}{\prod_{i,j} n\,_{i,j}!} \;
p\,_{x_1,N}$.

We say that an $I$--valued process $X=\{X_n\}_{n \in \mathbb{N}}$
is ME if $(X_1, \ldots , X_n)$ is ME for every $n$. In
\cite{DiaconisFreedman80MC} it is demonstrated that a recurrent
process ($X_1=X_n$ i.o.) is ME if, and only if, its law is a
mixture of Markov Chains. That is, let $\mathcal{P}$ be the space
of all  the stochastic  matrices $\Theta=\{\theta\,_{i,j}\}_{i,j}$
on $I \times I$. Then there exists, and is unique, a mixing
measure $\nu$ on the Borel sets of $I\times \mathcal{P}$ such that
\[
P(X_1=x_1,\ldots, X_n=x_n) = \int_{\mathcal{P}} \prod_{i=1}^{n-1}
\theta_{x_i,x_{i+1}} \: \nu (x_1,d \mathbf{\Theta})
\]
Let $\Gamma_i(k)$ be the step of the process $X$ at which the
state $i$  occurs for the $k$--th time. Let $V_i(k)$ be the
$k$--th successor of the state $i$, i.e. the variable immediately
subsequent the $k$--th occurrence of $i$
($V_i(k)=X_{\Gamma_i(k)+1}$). The hypothesis of de~Finetti  was
that, if all the subsequences $\{V_i(k)\}_{k=1, \ldots , n_i^+}$,
for $i \in I$, are exchangeable and $\infty$--extendible, then $X$
is a mixture of Markov Chains. This actually occurs if all the
states in $I$ are recurrent, but Lemma 5 in \cite{Fortinietal02}
assures that a recurrent ME process is strongly recurrent, then
all the states are recurrent and the two characterizations
coincide.

Zaman in \cite{Zaman84, Zaman86} demonstrated that finite Markov
exchangeability does not coincide with finite exchangeability of
the $\{V_i(k)\}_{k=1, \ldots , n_i^+}$, $i \in I$. In fact, given
$x_1$ and $N$, some of the transitions in $(X_1, \ldots , X_n)$
should necessarily occur as last. Then, the subsequences
$\{V_i(k)\}_k$ are invariant only under permutations that do not
alter those forced transitions. Zaman described the extremal
$n$--ME distributions as particular urn processes without
replacement where some balls should necessarily be drawn as last,
but the characterization of the mixture of Markov Chains cannot be
derived through a passage to the limit without adding some
restrictions.

\subsection{\textmd{Markov exchangeable binary sequences}}

The proofs of the theorems \eqref{thm: no of matrices},
\eqref{thm: from mark w to p} and \eqref{thm: from mark p to p} of
this section are in appendix.

If  $I=\{0,1\}$ we deal with  $2\times 2$ transition count
matrices of the kind:
\[
N=\left(%
\begin{array}{cc}
  n\,_{0,0} & n\,_{0,1} \\
  n\,_{1,0} & n\,_{1,1} \\
\end{array}%
\right)
\]
and it is $n\,_{0,0}+n\,_{0,1}=n_0^+$ and
$n\,_{1,0}+n\,_{1,1}=n_1^+$. The term $\det(B_{x_n, x_n})$ in
Theorem \ref{thm: Whittle} simply is $n\,_{1,0}/n_1^+$ if $x_n=0$,
and $n\,_{0,1}/n_0^+$ if $x_n=1$. So we have
\begin{equation}\label{eq : p w mark exb 01}
w\,_{x_1, N} = \left\{
\begin{array}{rl}
\binom{n_0^+}{n\,_{0,0}}      \binom{n_1^+-1}{n\,_{1,1}}  \; p\,_{x_1, N} & \quad \mbox{ if } (x_1,N) \mbox{ imply } x_n=0 \\
\binom{n_0^+-1}{n\,_{0,0}}    \binom{n_1^+}{n\,_{1,1}}    \; p\,_{x_1, N} & \quad \mbox{ if } (x_1,N) \mbox{ imply } x_n=1 \\
\end{array}
\right.
\end{equation}
We can consider separately the sequences depending on the initial
state. From now on, we  fix $P(X_1=0)=1$ and hence we will
consider only the sequences starting with 0 and  the probabilities
$\{w\,_{0,N}\}_{N \in \Phi(0,n)}$, and $\{p\,_{0,N}\}_{N \in
\Phi(0,n)}$.  We will also use the self--explaining notation
$p\,_0 \smx{n\,_{0,0}&n\,_{0,1}\\n\,_{1,0} & n\,_{1,1} }$ and
$w\,_0 \smx{n\,_{0,0}&n\,_{0,1}\\n\,_{1,0} & n\,_{1,1} }$ when we
need to display the number of transitions.

Unlike the DFPE case,  the number of  probabilities defining an
$n$--ME distribution is not so evident. We have to count the
possible different transition count matrices for each fixed
starting state. From \eqref{eq: 1 kind} and \eqref{eq: 2 kind} two
cases are possible when $X_1=0$, and we define
\begin{gather*}
\Phi_1(0,n) = \{N \in \Phi(0,n) \: : \: n\,_{0,1} = n\,_{1,0} \}\\
\Phi_2(0,n) = \{N \in \Phi(0,n) \: : \: n\,_{0,1} = n\,_{1,0}+1 \}
\end{gather*}
such that $\Phi_1(0,n) \cup \Phi_2(0,n) =\Phi(0,n)$. Call the
transition count matrices of $\Phi_1(0,n)$ matrices of the first
kind, and those of $\Phi_2(0,n)$ of the second kind.  The
following theorem corrects the assertion $|\Phi(0,n)| = 1 +
\binom{n-1}{2}$ stated in a different form in \cite[page
239]{DiaconisFreedmanKoch} and reported in \cite{Jeffrey04}
\begin{thm}\label{thm: no of matrices}
\[
|\Phi(0,n)| = 1 + \binom{n}{2}
\]
\end{thm}
For symmetry reasons the same result is valid for the sequences
starting in 1.

Now we state a couple of equations we will use in the following.
For any $n$ and $N$ we have
\begin{align}
p\,_{0,N} &= p\,_0 \smx{
 n\,_{0,0} & n\,_{0,1}\\
 n\,_{1,0} & n\,_{1,1}
}   = p\,_0 \smx{
 n\,_{0,0}+1 & n\,_{0,1}\\
 n\,_{1,0}   & n\,_{1,1}
} +
 p\,_0 \smx{
 n\,_{0,0} & n\,_{0,1}+1\\
 n\,_{1,0} & n\,_{1,1}
} & \text{ if } N \in \Phi_1(0,n)   \label{eq: mark p 0}\\
p\,_{0,N} &= p\,_0 \smx{
 n\,_{0,0} & n\,_{0,1}\\
 n\,_{1,0} & n\,_{1,1}
}   = p\,_0 \smx{
 n\,_{0,0}   & n\,_{0,1}\\
 n\,_{1,0}+1 & n\,_{1,1}
} +
 p\,_0 \smx{
 n\,_{0,0} & n\,_{0,1}  \\
 n\,_{1,0} & n\,_{1,1}+1
} & \text{ if } N \in \Phi_2(0,n)  \label{eq: mark p 1}
\end{align}

The first $k$ steps $(X_1,\ldots , X_k)$, $k<n$, of an $n$--ME
sequence are $k$--ME, and we can obtain all the probabilities
$\{p\,_{0,K}\}_{K \in \Phi(0,k)}$ from the $\{p\,_{0,N}\}_{N \in
\Phi(0,n)}$. Let $K=\smx{
  k\,_{0,0} & k\,_{0,1}\\
  k\,_{1,0} & k\,_{1,1}
}$ be the transition count matrix up to step $k$ of a sequence
starting in 0, and let $k\,_{0,0}+k\,_{0,1}=k_0^+$ and
$k\,_{1,0}+k\,_{1,1}=k_1^+$. Then
\begin{thm}\label{thm: from mark p to p}
\begin{multline*}
p\,_{0,K} = \sum_{N \in \Phi_1(0,n)}
\frac{(n\,_{0,0})_{k\,_{0,0}}(n\,_{0,1})_{k\,_{0,1}}
}{(n_0^+)_{k_0^+ }}
\frac{  (n\,_{1,1})_{k\,_{1,1}}  (n\,_{1,0}-1)_{k\,_{1,0}} }{(n_1^+-1)_{ k_1^+ }}\;\;w\,_{0,N}+\\
\qquad\quad + \sum_{N \in \Phi_2(0,n)} \frac{
(n\,_{0,0})_{k\,_{0,0}} (n\,_{0,1}-1)_{k\,_{0,1}} }{ (n_0^+ -1)_{
k_0^+ } } \frac{ (n\,_{1,1})_{k\,_{1,1}} (n\,_{1,0})_{k\,_{1,0}}
}{(n_1^+)_{ k_1^+ }}\;\;w\,_{0,N}
\end{multline*}
\end{thm}
where the sums should be restricted over those matrices $N$ in
$\Phi(0,n)$ having $n\,_{i,j} \geq k\,_{i,j}$, for all $i$, $j$ in
$\{0,1\}$. Consider  the probability $p\,_0 \smx{
  a & 1\\
  0 & b
}$ of having the sequence of $a+b+2$ steps starting in 0 with $a$
transitions $(0,0)$, a single transition $(0,1)$ and ending with
$b$ transitions $(1,1)$, and denote it $w\,_{0,a,b}$. By the above
theorem we have
\begin{equation}\label{eq: markq markl}
\begin{split}
w\,_{0,a,b} =p\,_0 \smx{a & 1\\0 & b }=
&\sum_{N \in\Phi_1(0,n)}\frac{ (n\,_{0,0})_{a} \; n\,_{0,1}}{ (n_0^+)_{a+1} }\frac{(n\,_{1,1})_{b}}{(n_1^+-1)_{b}}\;\;w\,_{0,N}+\\
&+ \sum_{N \in \Phi_2(0,n)} \frac{ (n\,_{0,0})_{a} \;
(n\,_{0,1}-1) } { (n_0^+-1)_{a+1} }
\frac{(n\,_{1,1})_{b}}{(n_1^+)_{b}}\;\; w\,_{0,N}
\end{split}
\end{equation}
We set $w\,_{0,n-1,0} = p\,_0 \smx{n-1 & 0\\0 & 0}$.

Define the operators $\Delta_0$ and $\Delta_1$  such that:
\[
\Delta_0\left(w\,_{0,a,b} \right) =  w\,_{0,a+1,b}  - w\,_{0,a,b}
\qquad \mbox{ and } \qquad \Delta_1\left(w\,_{0,a,b} \right) =
w\,_{0,a,b+1}  - w\,_{0,a,b}
\]
Then we have
\begin{thm}\label{thm: from mark w to p}
\[
p\,_{0,N}  = (-1)^{n\,_{0,1}-1 + n\,_{1,0}}\Delta_0^{n\,_{0,1}-1}
\: \Delta_1^{n\,_{1,0}} \left(w\,_{0,n\,_{0,0},n\,_{1,1}}\right)
\]
\end{thm}
In an $n$--ME sequence the probabilities $\{w\,_{0,a,b}\}$ are
well defined for every couple of nonnegative integers $(a,b)$
having sum not greater than $n-2$, together with the case
$w\,_{0,n-1,0}$. Denote as $\mathcal{L}_n$ the set of couples
$(a,b)$ such defined together with the couple $(n-1,0)$. Theorem
\ref{thm: from mark w to p} assures that the probabilities
$\{w\,_{0,a,b}\}_{\mathcal{L}_n}$ suffice to completely define an
$n$--ME sequence starting in 0. It is easily seen that
$|\mathcal{L}_n|$ = $\binom{n}{2}+1$ as we would expect.

\subsection{\textmd{Extendibility}}

Unlike the DFPE case, in a ME sequence it is meaningless to
consider separately the extendibility of the two  subsequences
$\{V_0(k)\}_k$ and $\{V_1(k)\}_k$. Then we say that an $n$--ME
sequence $(X_1,\ldots ,X_n)$ is $r$--extendible if there exist
$(X_{n+1},\ldots ,X_r)$ such that $(X_1,\ldots, X_r)$ is $r$--ME.

The probabilities $w\,_{0,a,b}$ allow us to study the
extendibility of a ME sequence in a geometric approach analogous
to that of Section \ref{sec: ext part}.

The space of the probabilities $\{w\,_{0,a,b}\}_{\mathcal{L}_n}$
of all the $n$--ME sequences starting in 0 (call it $\Gamma_n$) is
implicitly defined by Theorem \ref{thm: from mark w to p}. That
is, we have that every $w\,_{0,a,b}$ should satisfy
\begin{equation}\label{eq: conditions Markovl}
(-1)^{c+d-1}\Delta_0^{c-1} \Delta_1^d(w\,_{0,a,b})\geq 0 \qquad
\forall (c,d)\: : \; \smx{a & c\\d & b } \in \Phi(0,n)
\end{equation}
and we can write
\[
\Gamma_n=\Big\{(w\,_{0,a,b})_{\mathcal{L}_n} \: : \;
w\,_{0,a,b}\geq 0, \; \mbox{\eqref{eq: conditions Markovl} is
satisfied} \Big\}
\]
Theorem \ref{thm: from mark w to p} and  \eqref{eq: markq markl}
establish affine congruence between the unitary
$\binom{n}{2}$--dimensional simplex
$\diamondsuit_{\binom{n}{2}+1}$, which is the space of the
probabilities $\{w\,_{0,N}\}_{N \in \Phi(0,n)}$, and $\Gamma_n$,
which consequently is a $\binom{n}{2}$--dimensional (non standard)
simplex. The vertices of $\diamondsuit_{\binom{n}{2}+1}$ represent
the extremal distributions of Theorem \ref{thm:finite partial
exb}. Equation \eqref{eq: markq markl} maps them   to the vertices
of  $\Gamma_n$. We will denote as $\gamma_N$ the vertex of
$\Gamma_n$ corresponding to the extremal distribution having
$w\,_{0,N}=1$.

An $n$--ME sequence starting in 0 represented in $\Gamma_n$ by the
point $(w\,_{0,a,b})_{\mathcal{L}_n}$ is $r$--extendible if, and
only if, there exist probabilities $\{w\,_{0,a,b}\}$ with $(n-2) <
(a+b) \leq (r-2)$ together with $w\,_{0,r-1,0}$, such that
$(w\,_{0,a,b})_{\mathcal{L}_r}$ lies in $\Gamma_r$.  Let
$\Gamma_{r}^{(n)}$ be the orthogonal projection of $\Gamma_r$ over
the coordinates of $\Gamma_n$, and let $\gamma_R^{(n)}$ be the
analogous projection of $\gamma_R$. Then $\Gamma_{r}^{(n)}$
represents the $n$--ME sequences that are (at least)
$r$--extendible and is the convex hull of the
$\{\gamma_R^{(n)}\}_{R \in \Phi(0,r)}$.

By \eqref{eq : p w mark exb 01}, \eqref{eq: mark p 0}, \eqref{eq:
mark p 1},  \eqref{eq: markq markl}, and with passages similar to
those of the proof of Theorem \ref{thm: representation partiall}
one can prove the following is valid for any $r>n$:
\begin{equation}\label{eq: relation points markovl}
\begin{split}
  \gamma^{(n)}_{R} = \gamma^{(n)}\smx{
  r\,_{0,0} & r\,_{0,1}\\
  r\,_{1,0} & r\,_{1,1}
} = \frac{r\,_{0,1}}{r_0^+} \; \gamma^{(n)}\smx{
  r\,_{0,0}   & r\,_{0,1}\\
  r\,_{1,0}-1 & r\,_{1,1}
} + \frac{r\,_{0,0}}{r_0^+} \; \gamma^{(n)}\smx{
  r\,_{0,0}-1 & r\,_{0,1}\\
  r\,_{1,0}   & r\,_{1,1}
} \quad \mbox{ if } R \in \Phi_1(0,r)\\
\gamma^{(n)}_{R} = \gamma^{(n)}\smx{
  r\,_{0,0} & r\,_{0,1}\\
  r\,_{1,0} & r\,_{1,1}
} = \frac{r\,_{1,0}}{r_1^+} \; \gamma^{(n)}\smx{
  r\,_{0,0} & r\,_{0,1}-1\\
  r\,_{1,0} & r\,_{1,1}
} + \frac{r\,_{1,1}}{r_1^+} \; \gamma^{(n)}\smx{
  r\,_{0,0} & r\,_{0,1}\\
  r\,_{1,0} & r\,_{1,1}-1
} \quad \mbox{ if } R \in \Phi_2(0,r)
\end{split}
\end{equation}
As a consequence, $\Gamma_{r+1}^{(n)}$ is embedded in
$\Gamma_{r}^{(n)}$ and $\{ \Gamma_{r}^{(n)} \}_r$ is a nested
sequence of convex polytopes.  To verify whether a point
representing a distribution lies inside a certain polytope, and
establish its extendibility, we can use a linear program analogous
to  \eqref{eq: LP}.

We have computationally calculated the volume of some of the
polytopes $\Gamma_{r}^{(n)}$. We consider the ratio of the volume
of $\Gamma_{r}^{(n)}$ to the volume of $\Gamma_n$ as an index of
the proportion of $n$--ME distribution that are $r$--extendible,
as has been done in  \cite{Crisma82} and \cite{Wood92} for the
exchangeable case, and we report some values in Table \ref{table}.
By \eqref{eq: relation points markovl} one can see that, unlike
the DFPE case, not all the points $\gamma_{R}^{(n)}$ are vertices
of $\Gamma_{r}^{(n)}$ as some of them are redundant. A strange
consequence is that $\Gamma_{r}^{(3)}=\Gamma_3$ for any $r$, so in
Table \ref{table} we start with $n=4$.
\begin{table}[h!]
\[
\begin{array}{c|c c c c c c}
 n \; \backslash \; r & 5 & 6 & 7 & 8 & 9 & 10 \\
 \hline
  4 & 0.75 & 0.6667 & 0.6024 & 0.5504 & 0.5105 & 0.4778 \\
  5 &      & 0.4445 & 0.2860 & 0.2018 & 0.1454 & 0.1091 \\
  6 &      &        & 0.1929 & 0.0738 & 0.0336 &        \\
  7 &      &        &        & 0.0625 & 0.0111 &        \\
  8 &      &        &        &        & 0.0146 &        \\
  9 &      &        &        &        &        & 0.0025 \\
\end{array}
\]
\caption{Values of $Vol(\Gamma_r^{(n)})$ / $Vol(\Gamma_n)$ for
different values of  $n$ and $r$. The entries relative to
$n=6,7,8$ with $r=10$ are missing since it seems computationally
intractable to find the relative volume of
$\Gamma_r^{(n)}$.}\label{table}
\end{table}

\subsubsection{$\infty$--extendible case}

An $\infty$--extendible $n$--ME sequence is not necessarily the
initial segment of a mixture of Markov Chains. As pointed out in
\cite{DiaconisFreedman80MC}, an infinite ME  sequence starting in
0 is a mixture of two kinds of processes: recurrent Markov Chains
and processes that deterministically begin with a streak of zeros,
make a single $(0,1)$ transition and end with all ones. But if, as
$n \to \infty$, both $n^+_0$ and $n^+_1$ go to infinity,  there
exists a unique mixing measure $\nu$ over $[0,1]^2$ and a couple
$(\theta\,_{0,0},\theta\,_{1,1})$ such that, conditionally on
$X_1=0$
\[
p_0\smx{n\,_{0,0}&n\,_{0,1}\\n\,_{1,0} & n\,_{1,1} } = \int_0^1
\int_0^1 \theta\,_{0,0}^{n\,_{0,0}}(1-\theta\,_{0,0})^{n\,_{0,1}}
\: \theta\,_{1,1}^{n\,_{1,1}}(1-\theta\,_{1,1})^{n\,_{1,0}}  d
\nu\: (\theta\,_{0,0}, \theta\,_{1,1})
\]

Denote the indicator function of the event $\big\{$the $k$--th
successor of i is j$\big\}$ as $Y_{i,j}(k)$: $\mathbbm{1}\,_{j}
\big(V_i(k) \big) = Y_{i,j}(k)$ for all $i,j$ in $\{0,1\}$. Then
we can write
\[
w\,_{0,a,b} = E\left[(1-X_1) \cdot Y_{0,0}(1)\cdots Y_{0,0}(a)
\cdot \Big(1-Y_{0,0}(a+1)\Big) \cdot Y_{1,1}(1)\cdots
Y_{1,1}(b)\right]
\]
When we consider sequences starting both in 0 and 1, we introduce
the probabilities  $w\,_{1,a,b}$, defined as the probabilities of
having the sequence starting in 1 with $b$ transitions $(1,1)$ a
single transition $(1,0)$, and ending with $a$ transitions
$(0,0)$. Then we have:
\[
w\,_{1,a,b} = E\left[X_1 \cdot Y_{1,1}(1)\cdots Y_{1,1}(b) \cdot
\Big(1-Y_{1,1}(b+1)\Big) \cdot Y_{0,0}(1)\cdots Y_{0,0}(a)\right]
\]
In a  mixture of Markov Chains it is
\[
\begin{array}{l}
  w\,_{0,a,b} =  E_{\nu} \left[(1-X_1) \; (\theta\,_{0,0})^{a}
(\theta\,_{1,1})^{b} \right] - E_{\nu} \left[(1-X_1) \;
(\theta\,_{0,0})^{a+1} (\theta\,_{1,1})^{b}
\right] \\
  w\,_{1,a,b} =  E_{\nu} \left[X_1 \; (\theta\,_{0,0})^{a}
(\theta\,_{1,1})^{b} \right] - E_{\nu} \left[X_1 \;
(\theta\,_{0,0})^{a} (\theta\,_{1,1})^{b+1}
\right] \\
\end{array}
\]
So, unlike the DFPE case, we do not have the mixed moments of the
mixing distribution, but those involved differences. It is easily
seen that it is not possible to single out  the mixed moments from
the probabilities $w\,_{0,a,b}$ and $w\,_{1,a,b}$. However, let
$\textbf{N}$ be the transition count matrix of $(X_1,\ldots ,X_n)$
intended as a r.v. Then, if the ME distribution is such that $X_1$
and $\textbf{N}$ are independent, we can obtain them. Define
\[
m\,_{a,b} = E\left[Y_{0,0}(1)\cdots Y_{0,0}(a) \cdot
Y_{1,1}(1)\cdots Y_{1,1}(b)\right]
\]
and let $P(X_1=i)=q_i$. Under  independence of $X_1$ and
$\textbf{N}$ we have
\[
\frac{ w\,_{ 0,a,b } }{ q_0} = m\,_{a,b} - m\,_{a+1,b} \qquad
\mbox{ and } \qquad \frac{ p\,_1 \smx{ 0 & 0\\ 0 & b} }{q_1} =
m\,_{0, b}
\]
Then we have $m\,_{1,b}= m\,_{0,b} -
\left(w\,_{0,0,b}/q_0\right)$, and in general, by recurrence
\[
m\,_{a,b}= m\,_{a-1,b} - \frac{w\,_{0,a-1,b}}{q_0}
\]
So an $n$--ME distribution such that $X_1$ and $\textbf{N}$ are
independent is defined by the quantities $m\,_{a,b}$, for every
couple $(a,b)$ such that $a+b \leq n-1$. In  a mixture of Markov
Chains it is $m\,_{a,b}$ = $E_{\nu}\left[\theta\,_{0,0}^a \:
\theta\,_{1,1}^b\right]$ and we can formulate generalized
covariances as in \eqref{eq: generalized covariance} and
\eqref{eq: partiall partialcov} and state simple necessary
conditions for $\infty$--extendibility as in the DFPE case.

\section{{\large Concluding remarks}}

For what we have said, either for exchangeable, DFPE and ME cases,
the $\infty$--extendible sequences are a particular subset of all
the sequences of a fixed length. Then, in the inferential analysis
of  binary data, one can look for distributions which do not need
the assumption of $\infty$--extendibility as an alternative to the
mixtures of i.i.d and mixtures of Markov Chains processes. So, a
preliminary analysis of the extendibility of the data at hand
(i.e. of their empirical distribution) can give some evidences
against a mixture model, and the present paper give the tools for
this purpose.

Gupta in \cite{Gupta99, Gupta00} looked for an extension of the
Hausdorff's moment problem for distributions over the simplex, and
implicitly  found the necessary and sufficient conditions for the
extendibility of an exchangeable finite sequence  taking values in
a finite state space,  with the same geometric interpretation we
have given. Combining his results with those of Section \ref{sec:
par exb} one can easily find the conditions for the extendibility
of DFPE sequences when the variables assume more than two values.
It seems hard to find an analogous extension for the ME case.

\section*{{\large Appendix A. Proof of Theorem \ref{thm: no of
matrices}}}

We first find $|\Phi_1(0,n)|$. In a sequence of length $n$ we have
$n-1$ transitions. For every fixed value for $n\,_{1,0}=n\,_{0,1}$
equal to $k$, say, the couple $(n\,_{1,1} , n\,_{0,0} )$  can
assume all the possible values such that $(n\,_{1,1}  + n\,_{0,0}
) = n-1-2k$, whose number is $(n-2k)$. The possible values for
$k=n\,_{1,0}=n\,_{0,1}$ range in $0,1,\ldots ,\lfloor (n-1)/2
\rfloor$, where $\lfloor (n-1)/2 \rfloor$ is the integer part of
$(n-1)/2$. In the special case $n\,_{1,0} = n\,_{0,1} = 0$ we have
only one matrix $\smx{n-1 & 0\\ 0 & 0}$. So we have
\[
|\Phi_1(0,n)| = 1 + \sum_{k=1}^{\lfloor (n-1)/2 \rfloor} (n -2k)
\]
Now consider the following two arguments:
\begin{itemize}
    \item All the sequences
consistent with a  matrix in $\Phi_2(0,n)$ start in $0$ and end in
$1$. If we add a transition $(1,0)$ at the end of any such
sequence, its transition count  matrix belong to $\Phi_1(0,n+1)$.
    \item If we reduce of one the number of transitions $(1,0)$ in a  matrix of $\Phi_1(0,n+\nolinebreak1)$,
we obtain a matrix of $\Phi_2(0,n)$.
\end{itemize}
Consequently, each matrix of the second kind is constructible by
one of the first kind of a step longer, as long as $n\,_{1,0}$ is
not null. Then we have to exclude the matrix having $n\,_{1,0} =
n\,_{0,1} = 0$ and it is:
\[
|\Phi_2(0,n)| = |\Phi_1(0,n+1)| -1 = \sum_{k=1}^{\lfloor n/2
\rfloor} ( n +1 -2k)
\]
Clearly it is $|\Phi(0,n)| = |\Phi_1(0,n)| + |\Phi_2(0,n)|$, that
is
\[
|\Phi(0,n)| = 1 + \sum_{k=1}^{\lfloor (n-1)/2 \rfloor} (n -2k) +
\sum_{k=1}^{\lfloor n/2 \rfloor}  (n -2k +1) = 1 +
\sum_{k=1}^{n-1} (n-k) = 1 + \binom{n}{2}
\]

\section*{{\large Appendix B. Proof of Theorem \ref{thm: from mark w to
p}}}

$w\,_{0,a,b}$  is the probability $p\,_0 \smx{a & 1\\ 0 & b}$ of
having a sequence starting in 0 and ending in 1. Then by
\eqref{eq: mark p 1} we have
\begin{equation*}
\begin{split}
w\,_{0,a,b} = p\,_0 \smx{a & 1\\ 0 & b} \;=\; & p\,_0 \smx{a & 1\\1 & b} + p\,_0 \smx{a & 1\\ 0 & b+1}   \\
                                        \;=\; & p\,_0 \smx{a & 1\\1 & b} + w\,_{0,a,b+1}
\end{split}
\end{equation*}
Then it follows that
\begin{equation}\label{delta1}
p\,_0 \smx{a & 1\\1 & b} = -\Delta_1\left(w\,_{0,a,b}\right)
\end{equation}
so, we can derive the probability of having any  sequence starting
in 0 and consistent with the transition count matrix $\smx{ a & 1\\
1 & b }$. These sequences end in 0, so by \eqref{eq: mark p 0} we
have
\begin{equation}\label{from lambda to q markov exb 2}
p\,_0 \smx{a & 2\\1 & b} = p\,_0 \smx{a & 1\\1 & b} - p\,_0
\smx{a+1 & 1\\1 & b}
\end{equation}
We have just demonstrated that all the terms on the right hand
side of \eqref{from lambda to q markov exb 2} can be derived from
\eqref{delta1}, and it is $p\,_0 \smx{a & 2\\1 & b} =
\Delta_0\left(\Delta_1\left(w\,_{0,a,b}\right)\right)$. So, we can
derive  the probability of any sequence starting in 0 and
consistent with the transition count matrix $\smx{ a & 2\\
1 & b }$. For an $n$--ME sequence starting in 0, it is always
$n\,_{0,1}=n\,_{1,0}$ or $n\,_{0,1}=n\,_{1,0}+1$. So, repeating
the previous passages, by recurrence, we obtain:
\[
p\,_{0,N} = p_0\smx{n\,_{0,0}&n\,_{0,1}\\n\,_{1,0} & n\,_{1,1} } =
\left\{\begin{array}{cc}
  -\Delta_1\Big(\Delta_0 \circ\Delta_1\Big)^{n\,_{1,0}-1} \left(w\,_{0,n\,_{0,0},n\,_{1,1}}\right) & \mbox{ if } N \in \Phi_1(0,n) \\
  \Big(\Delta_0 \circ\Delta_1\Big)^{n\,_{1,0}} \left(w\,_{0,n\,_{0,0},n\,_{1,1}}\right) & \mbox{ if } N \in \Phi_2(0,n) \\
\end{array}
\right.
\]
which is equivalent to Theorem \ref{thm: from mark w to p}.\\

\section*{{\large Appendix C. Proof of Theorem \ref{thm: from mark p to
p}}}

Let $K$ be the transition count matrix of the first  $k$ steps of
the sequence. The number of sequences $(x_1, \ldots ,x_n) \in
\{0,1\}^n$, with $x_1= 0$, such that $K=\smx{k\,_{0,0} &
k\,_{0,1}\\ k\,_{1,0} & k\,_{1,1} }$ and $N=\smx{n\,_{0,0} &
n\,_{0,1}\\ n\,_{1,0} & n\,_{1,1} }$ is equal to the number of
sequences consistent with the transition count matrix $ \smx{
  n\,_{0,0}-k\,_{0,0}\: & \:n\,_{0,1}-k\,_{0,1}\\
  n\,_{1,0}-k\,_{1,0}\: & \:n\,_{1,1}-k\,_{1,1}
}
$
that is
\[
\left\{
\begin{array}{ll}
\binom{n_0^+-k_0^+}{n\,_{0,0}-k\,_{0,0}}        \binom{n_1^+ - k_1^+ -1}{n\,_{1,1}-k\,_{1,1}}    \;  & \mbox{if } \; x_n=0\\
\binom{n_0^+-k_0^+ -1}{n\,_{0,0}-k\,_{0,0}}        \binom{n_1^+ - k_1^+}{n\,_{1,1}-k\,_{1,1}}    \;  & \mbox{if } \; x_n=1\\
\end{array}
\right.
\]
But, as we have said, since we have fixed $x_1=0$, it is $x_n=0$
if $N$ is of the first kind, and $x_n=1$ if $N$ is of the second
kind. Then it is
\begin{multline*}
p\,_{0,K} = \sum_{N \in \Phi_1(0,n)}
 \binom{n_0^+-k_0^+}{n\,_{0,0}-k\,_{0,0}}  \binom{n_1^+ - k_1^+ -1}{n\,_{1,1}-k\,_{1,1}}
 \;\; p\,_{0,N}    +\\
\quad \quad \quad\qquad \qquad \qquad \qquad\qquad \qquad+ \sum_{N
\in \Phi_2(0,n)} \binom{n_0^+-k_0^+-1}{n\,_{0,0}-k\,_{0,0}}
\binom{n_1^+ - k_1^+}{n\,_{1,1}-k\,_{1,1}}  \;\; p\,_{0,N}
\end{multline*}
Finally the theorem follows by \eqref{eq : p w mark exb 01}  and
the fact that
\[
\frac { \binom{ n_0^+-k_0^+ } { n\,_{0,0}-k\,_{0,0} } }
{\binom{n_0^+}{ n\,_{0,0} } } = \frac{(n\,_{0,0})_{k\,_{0,0}}
(n_0^+-n\,_{0,0})_{k_0^+-k\,_{0,0}}}{(n_0^+)_{k_0^+}}
\]

\end{document}